\documentclass{amsart}

\usepackage{epsfig}
\usepackage{amscd}
\usepackage{amsmath, amssymb, amsthm}
\usepackage{graphics}
\usepackage{color}
\usepackage{dsfont}
\usepackage[active]{srcltx}
\usepackage{soul}

\def\XXint#1#2#3{{\setbox0=\hbox{$#1{#2#3}{\int}$ }
\vcenter{\hbox{$#2#3$ }}\kern-.6\wd0}}

\newcommand{\blue}{\color{blue}}
\newcommand{\red}{\color{red}}

\newcommand{\edit}[1]{ {\red \cs #1 \cs}}

\newcommand{\cs}{$\clubsuit$}

\newcommand{\wt}{\widetilde}

\newcommand{\Z}{{\mathbb Z}}

\numberwithin{equation}{section}

\newtheorem{theo}{{\sc Theorem}}

\newtheorem{lem}[theo]{{\sc Lemma}}

\newenvironment{rem}{\medskip\noindent{\it Remark:\/} }{\medskip}

\newcommand{\R}{{\mathbb R}}

\newtheorem*{main-theorem}{Main Theorem}
\newtheorem{proposition}{Proposition}[section]
\newtheorem{theorem}{Theorem}[section]
\newtheorem*{old-thm}{Theorem}
\newtheorem{lemma}[proposition]{Lemma}
\newtheorem{corollary}[proposition]{Corollary}
\newtheorem{conjecture}[theorem]{Conjecture}

\theoremstyle{definition}

\numberwithin{equation}{section}

\def\11{\mathds{1}}

\def\Ci{{\mathcal C}^\infty}

\def\supp{\mathrm{supp}\,}

\def\phi{\varphi}

\def\be{\begin{eqnarray*}}
\def\ee{\end{eqnarray*}}
\def\ben{\begin{eqnarray}}
\def\een{\end{eqnarray}}

\def\lll{\left\langle}
\def\rrr{\right\rangle}

\def\L2R{L_{\text{Rest}}^2}

\def\tchi{\tilde{\chi}}

\newcommand{\comp}{\operatorname{comp}}

\newcommand{\mcal}{\mathcal{M}}

\newcommand{\qcal}{\mathcal{Q}}

\newcommand{\scal}{\mathcal{S}}
\newcommand{\tcal}{\mathcal{T}}

\newcommand{\e}{\epsilon}

\begin{document}
\title[Lower bounds for Cauchy data]{Lower bounds for Cauchy data on curves in a negatively curved surface}

\author{Jeffrey Galkowski}

\address{Department of Mathematics, University College, London}\email{j.galkowski@ucl.ac.uk}
\author{Steve Zelditch}
\address{Department of Mathematics, Northwestern  University,
Evanston, IL 60208-2370}  \email{zelditch@math.northwestern.edu}

\begin{abstract}
We prove a uniform lower bound on Cauchy data on an arbitrary curve on a negatively curved
surface using the Dyatlov-Jin(-Nonnenmacher)  observability estimate on the global surface.  In the process, we prove some further results about defect measures of   restrictions of eigenfunctions  to a hypersurface.

\end{abstract}

\maketitle


\section{Introduction}
\label{introduction}

The purpose of this article is to prove  a positive lower bound for $L^2$ norms of  Cauchy data of Laplace eigenfunctions
on curves of a negatively curved surface $(M,g)$ without boundary.   Let $\{\phi_j\}_{j=0}^{\infty}$ be an orthonormal basis of eigenfunctions
of the Laplacian, $$
 \begin{array}{l}
  -\Delta_g \phi_j = \lambda_j^2 \phi_j, \;\;\; \langle \phi_j, \phi_k \rangle = \delta_{jk} \\
 \end{array},$$
where $\langle f, g \rangle = \int_M f \bar{g} dV$ ($dV$ is the
volume form of the metric)   
Let  $H \subset M$ be a smooth hypersurface. 
The  semiclassical Cauchy data along $H$ is defined by 
 \begin{equation} \label{CD} CD(\phi_j)  := \{(\phi_j |_{H}, \;
\lambda_j^{-1}  D_{\nu} \phi_j |_{H}) \}
\end{equation}
where $D_{\nu}:=-i\partial_\nu$ and $\nu$ is a choice of unit normal to $H$.
For technical reasons, we also introduce 
the {\it renormalized Cauchy data} \begin{equation} \label{RCD}  RCD(\phi_j) = \{(1 +
 \lambda_j^{-2} \Delta_H) \phi_{j} |_H, h D_{\nu} \phi_j |_H) \}.  \end{equation} Here, $ \Delta_H$ denotes the negative
tangential Laplacian for the induced metric on $H$. Hence, as a semiclassical pseudodifferential operator,
the operator $(1+\lambda_j^{-2} \Delta_H)$ is characteristic precisely on the
glancing set $S^*H$ of $H$ and damps  out the whispering gallery
components of $\phi_j |_H$.
 
 In what follows, we use semi-classical notation $h_j = \lambda_j^{-1}$, since we will be using  semi-classical pseudo-differential calculus. Our first
 result gives a lower bound on the $L^2$ norm of \eqref{RCD} along $H$, and in fact a lower bound for more general `matrix elements' relative to semi-classical
 pseudo-differential operators $Op_h(a)$ of order $0$ on $L^2(H)$ with semi-classical symbol $a \in S^0_{cl}(B^*H)$. We denote by
  $$B^*H := \{(q, \xi) \in T^* H: |\xi|_H < 1\}$$ the
unit co-ball bundle of $H$;   $|\cdot|_H$ is the restriction of $g$ to $T^*H$.  We also denote by $\gamma_H f : = f |_H$ the
 restriction of $f \in C(M)$ to $H$.
\begin{theorem} \label{thm1}
\label{T:LefINTRO}  Let $(M,g)$ be a compact, negatively curved surface without boundary and $H \subset M$ a smooth
 curve. Then there exist $h_0 >0$ and $C_H>0$  so that,
 $$\begin{array}{l}
\|  (h_jD_\nu \phi_j) |_{H} \|_{L^2(H)}^2 + \|  \phi_{j} |_H \|_{L^2(H)}^2  \geq C_{H} > 0.
\end{array}$$ More generally, this inequality can be mircolocalized: for
$a \in S^0_{sc}(H)$, such that $Op_h(a)$ has principal symbol, $a_0(x',\xi')$, satisfying $a_0 \geq 0$ on $B^*H$ and $\supp a_0\cap B^*H\neq \emptyset$, there exist $h_0>0$ and $C_{a_0} >0$ so that for $0<h<h_0$,
\begin{multline*}
\lll Op_h(a) h_j D_\nu \phi_j |_{H} ,  h_j D_\nu \phi_j |_H \rrr_{L^2(H)}  \\+ \lll Op_h(a) (1 +
 h_j^2 \Delta_H) \phi_{j} |_H, \phi_{j} |_H
\rrr_{L^2(H)}  \geq C_{a_0} .
\end{multline*}
\
\end{theorem}

 Note that the  second inequality of Theorem \ref{thm1}  implies the same result for 
Cauchy data. 

\begin{corollary} \label{thm2}
\label{T:Lef} With the same assumptions and notations as in Theorem \ref{thm1}, $$\begin{array}{l}
\lll Op_h(a) h_j D_\nu \phi_j |_{H} ,  h_j D_\nu \phi_j |_H \rrr_{L^2(H)}  + \lll Op_h(a)  \phi_{j} |_H, \phi_{j} |_H
\rrr_{L^2(H)}\geq C_{a_0} > 0.\\ \\
\end{array}$$

\end{corollary}
Indeed, this follows from the statement  $1 \geq (1 - |\xi'|^2)$ on $B^*H$ relating
the symbols of $I$ and of $(I + h^2 \Delta_H)$. It is sufficient to consider this region since $\phi_j$ concentrates on it in a strong sense (see \cite{CHT15}).

\begin{rem}  
It is argued
 in \cite[p. 3063-3064]{BHT}  that the renormalized Dirichlet  data 
\eqref{RCD}

$$
\langle (1+h_j^2\Delta_H)\phi_j|_{H},\phi_j|_{H}\rangle_{L^2(H)}
$$
 is a closer analogue to the Neumann
 data $\|h_jD_\nu \phi_j\|_{L^2(H)}^2$ than is the traditional Dirichlet data.  In fact, one can see that $h_jD_\nu$ behaves similarly to $(1+h^2\Delta_H)_+^{\frac{1}{2}}u $~\cite{G16}.

\end{rem}

We give two  proofs of the result, by two rather different approaches whose contrast seems to be of some independent interest. Both 
are based on the Dyatlov-Jin(-Nonnenmacher)  observability estimate \cite{DJ17, DJN19}. The first proof is based on the
 Rellich identity of \cite{CTZ13} (adapted from \cite{Bu})   relating interior and restricted matrix elements and  microlocal lifts of eigenfunctions.  
The second is based on hyperbolic equations and is closely
related to results in \cite{GL17}. It yields the following version of Corollary \ref{thm2}.
\begin{theorem}
\label{t:CD}
Let $(M,g)$ be a negatively curved surface and $H\subset M$ a smooth curve. Then for all $0\neq b_0\in C_c^\infty(B^*H)$, 
 there is $C_{b_0}>0$ and $h_0>0$ such that if 
$$
\|(-h_j^2\Delta_g -1) \phi_j\|_{L^2}=o\big(\frac{h_j}{\log h_j^{-1}}\big)\|\phi_j\|_{L^2},
$$
then for $0<h_j<h_0$
\begin{equation}
\label{e:claim}
0<C_{b_0}\|\phi_j \|_{L^2} <\|Op_h(b_0) \phi_j|_{H}\|_{L^2(H)}+\|Op_h(b_0)h_j \partial_\nu \phi_j |_{H}\|_{L^2(H)}.
\end{equation}
Moreover, for all $U\subset H$ open, there is $c>0$ such that for all $\phi_j$ satisfying
$$
(-h_j^2\Delta_g-1)\phi_j=0,
$$
we have 
\begin{equation}
\label{e:lowerBound}
0<c\|\phi_j\|_{L^2(M)}<\|\phi_j|_{H}\|_{L^2(U)}+\|h_j\partial_\nu \phi_j|_{H}\|_{L^2(U)}.
\end{equation}
\end{theorem}

In Section \ref{FURTHER}, we state some further results on microlocal defect measures $\mu$, in particular on the possible case where
$\mu (S^*_H M) > 0$. Here (and hereafter) $S^*_H M = \{(x, \xi) \in S^*M: x \in H \}$.

\subsection{Background} 
 
 The first  proof of Theorem \ref{thm1} develops the Rellich identity approach of 
  \cite{CTZ13} . In that article, it   is proved that a sequence of  renormalized Cauchy data  \eqref{RCD} of eigenfunctions  is 
 quantum ergodic along any hypersurface $H \subset M$ if the sequence of
eigenfunctions is quantum ergodic on the global manifold $M$. 
It is not known whether the full orthonormal basis of eigenfunctions of a negatively curved compact surface is quantum ergodic  (QE), but
a recent result of  Dyatlov-Jin(-Nonnenmacher)  shows that its microlocal defect measures must have {\it full support}, that is, they charge (give positive mass to) any open set. 
The Dyatlov-Jin(-Nonnenmacher) theorem is a 
microlocal observability estimate for quasimodes $u$ of compact hyperbolic (or more generally negatively curved)  surfaces:
 For all $u\in H^2(M)$, and all $a \in C_0^{\infty}(T^*M)$ with $a \geq 0$,
 $a$ not identically zero on $S^*M$,  there exists a
 constant $C(a) > 0$ so that
\begin{equation} \label{DJLB} || u||_{L^2} \leq C(a) ||Op_{h}(a) u||_{L^2}
+ \frac{C(a) \log(1/h)}{  h } ||(-h^2 \Delta - I) u||_{L^2}
. \end{equation}
In particular, if $u$ is an eigenfunction the second term is zero and one
has a lower bound on $\|Op_h(a)u\|_{L^2(M)}$.
 Theorem \ref{thm1} gives  a similar full support property for the restricted  microlocal defect measures of the Cauchy data on
a curve.

The results may be stated in terms of
microlocal defect measures (quantum limits). Several are involved: the
global microlocal defect measures on $M$, and restricted microlocal
defect measures on $H$, both for Neumann data, Dirichlet data and for
renormalized Dirichlet data. 

We denote by $p(x, \xi)= |\xi|_g^2$ the principal symbol of the Laplacian  $-\Delta_g$. We denote by $H_p$ its Hamilton vector field, and  by  $\phi_t:=\exp(tH_{|\xi|_g^2})$ its Hamiltonian flow, i.e. the geodesic flow $\phi_t: S^*M \to S^*M$  of $(M, g)$ on its unit cosphere bundle.

We  use  some notation and background
on  the semiclassical calculus of pseudo-differential operators as in the references \cite{Bu,DZ,CTZ13,Zw}. On both $H$ and $M$ we fix  (Weyl) quantizations $a \to a^w$ of semi-classical symbols to
semi-classical pseudo-differential operators. 
 When it is necessary to indicate  which manifold is involved, we use
 Fermi coordinates $(x', x_n) $ with $x'$ coordinates on $H$ (Section \ref{LAPSECT}), and we use capital letters $A^w(x, h D)$ to indicate operators on $M$ and
small letters $Op_h(a) = a^w(x', h D_{x'})$ to indicate  semi-classical 
pseudo-differential operators on $H$ (denoted by $\Psi_{sc}^*(H)$).

We recall that microlocal defect measures  are the semi-classical  limits of the functionals
$a \to \langle a^w(x,h_jD_x) u_j, u_j \rangle$, which are  often referred to as microlocal lifts or   Wigner distributions. 
 Let $\qcal^*$ denote the set of all microlocal defect measures for an orthonormal basis $\{\phi_j\}$ of eigenfunctions. That is, $\mu \in \qcal^*$ if
there exists a subsequence $\scal = \{\phi_{j_k}\}$ of eigenfunctions for
which $\langle A \phi_{j_k}, \phi_{j_k} \rangle \to \int_{S^*M} \sigma_A d\mu$. The microlocal defect measures on the global manifold $M$ are invariant
probability measures for the geodesic flow $\phi_t: S^*M \to S^*M$ (see e.g.~\cite[Chapter 5]{Zw}).
Throughout, we denote by $\mcal(X)$ the space of positive measures on a metric space $X$, $\mcal_1(X)$ the space of probability
measures and $\mcal_I(X)$ the space of invariant measures under a flow on $X$.

The Dyatlov-Jin observability estimate for eigenfunctions implies the following: \begin{theo} Let $(M,g)$ be a compact negatively curved surface without boundary, let $\{\phi_j\}$ be any choice of orthonormal basis of eigenfunctions,
and let $A \in \Psi^0(M)$ be a pseudo-differential operator of order zero with
a  non-negative  principal symbol $\sigma_A$ not vanishing identically on $S^*M$. Then,
\begin{equation} \label{INF} \inf_{\mu \in \qcal^*} \int_{S^*M} \sigma_A d\mu \geq C_A > 0. \end{equation}
\end{theo}

From \eqref{INF}, it follows  that all microlocal defect measures of 
eigenfunctions of compact negatively curved  surfaces have {\it full support}, 
i.e. charge every open set of $S^*M$.

Given a quantization $a \to Op_h(a)$  of semi-classical symbols $ a \in S^{0}_{sc}(H)$ of order zero (see \cite{Zw}) to semi-classical pseudo-differential operators on $L^2(H)$,
we  define the microlocal lifts of the
Neumann data as the  linear functionals on $a \in S^{0}_{sc}(H)$ given by 
$$\mu_j^N(a): = \int_{B^* H} a  \, d\Phi_j^N : = \langle Op_h(a) h_j D_{\nu} \phi_j
|_{H}, h_j D_{\nu} \phi_j |_{H}\rangle_{L^2(H)}.  $$ 
We define the microlocal lifts of the Dirichlet date by
$$\mu_j^{D}(a): = \int_{B^* H} a \, d\Phi_j^{D} : = \langle Op_h(a) \phi_{j} |_H,  \phi_{j}|_H \rangle_{L^2(H)}.
$$
also  define the microlocal lifts of the modified Dirichlet
data by
$$\mu_j^{RD}(a): = \int_{B^* H} a \, d\Phi_j^{MD} : = \langle Op_h(a) (1 +
h_j^2 \Delta_H) \phi_{j} |_H,  \phi_{j}|_H \rangle_{L^2(H)}.
$$
Finally, we  define the microlocal lift $d \Phi_j^{RCD}$ of the renormalized Cauchy data  to be the sum
\begin{equation} \label{WIGCD} d \Phi_j^{RCD} := d \Phi_j^N + d \Phi_j^{RD}. \end{equation}
The weak* limits of the above microlocal lifts are termed microlocal defect
measures, respectively, of the Neumann, Dirichlet or renormalized Dirichlet type.

The distributions $\mu_j^N, \mu_j^D$  are asymptotically positive, but are not
normalized to have mass one and may tend to infinity.  
They  depend on the choice of quantization, but their possible weak* limits as $h_j \to 0$
do not, and the results of the article are valid for any choice of  quantization. We refer to  \cite{Zw} for background on semi-classical microlocal analysis.

\begin{theorem} \label{thm3}
\label{T:LefINTRO2} Let $(M,g)$ be a compact  surface without boundary of negative curvature. Suppose $H \subset M$ is a  curve. Then, for any
 $a \in S^0_{sc}(H),$ 
$$\begin{array}{l}  \inf_{\mu \in \qcal^*(RCD)} 
 \int_{B^*H}a d\mu  \geq C_{ a_0} > 0
\end{array}$$
and therefore, 
$$\begin{array}{l}  \inf_{\mu \in \qcal^*(CD)} 
 \int_{B^*H}a d\mu  \geq C_{ a_0} > 0
\end{array}$$
\end{theorem}

 One should be aware of an  obstruction to obtaining lower bounds on \eqref{RCD} from  lower bounds on \eqref{CD} for hypersurfaces
 in a general Riemannian manifold. Suppose that the  eigenfunctions $\phi_j$ are the
 highest weight spherical harmonics $Y_N^N(\theta, y) \simeq e^{i N \theta} e^{- N y^2/2}$ on $S^2$ along the equator $\gamma$. 
 The normal $y$-derivative vanishes since $Y_N^N$ is even under reflection across $\gamma$, i.e. $y \to - y$ in Fermi normal coordinates.
 But the restriction is $e^{i N \theta}$ and it is killed by $(I + N^{-2} \frac{\partial^2}{\partial \theta^2})$. Hence, the renormalized Cauchy
 data vanishes.  As this example shows, 
Rellich-type identities for renormalized Cauchy data {are not necessarily equivalent to bounds} on Cauchy data  \eqref{CD} because of the effects of
 concentration in tangential directions to $H$.  Theorem \ref{thm1} nevertheless gives a positive lower bound for curves on a negatively curved surface,
 because the Dyatlov-Jin lower bound \eqref{DJLB} implies that restrictions of eigenfunctions on such surfaces cannot concentrate entirely in the tangential
 directions.

\begin{rem} \label{INDINT} The proof of Theorem \ref{thm1} using the Rellich identity  is a continuation of the proof in \cite{CTZ13} in the case where
the microlocal defect measure of a sequence is Liouville measure $\mu_L$. Since the calculations in the case of a general microlocal defect
measure of independent interest, we present most of the details in all dimensions and in more detail than is strictly necessary for the
proof of Theorem \ref{thm1}. We only specialize to curves in a negatively curved surface at the end of the proof.

\end{rem}

\subsection{\label{FURTHER} Results on microlocal defect measures}

In the process of proving Theorem~\ref{thm1} via the Rellich formula we will obtain some facts about the collection of defect measures for eigenfunctions that are of independent interest. First, we study the restrictions of $\mu\in \mathcal{M}_I\cap \mathcal{M}_1$ to a hypersurface $H$
\begin{theorem}
\label{t:invMeas}
 Let $H\subset M$ be a hypersurface and $\mu\in \mathcal{M}_1(S^*M)\cap \mathcal{M}_I(S^*M)$ where the relevant flow is $\phi_t$. Then,
 \begin{itemize} \item $\mu |_{S^*_H M} = \mu |_{S^*H} $; 
 \item   the support of $\mu |_{S^*H}$ is  
    is contained in the null-space of the second fundamental form $Q(0, x', \xi')$, i.e. $Q(0, x', \xi')= 0$ for $(x', \xi') \in \rm{Supp} (\mu |_{S^*H}). $ 
    \item  Hence, 
    $\mu |_{S^*H}$ is supported on the subset of $S^* H$ where the Hamilton vector field $H_p $ coincides with the Hamilton vector field   $H_{p |_{S^*H}}$ of
    the submanifold metric norm,    $g_H(x', \xi') = |\xi'|^2_{g_H}$.
    \end{itemize}
\end{theorem}
\noindent In fact, we will see in Section~\ref{MDSECT} that such measures have $\mu|_{S^*H}$ supported where $H_p$ is tangent to $H$ to infinite order. 

  The following immediate corollary was mentioned in \cite{TZ17}.
\begin{corollary} If the second fundamental form of $H$ is non-degenerate, $\mu |_{S^*_H M} = 0$. \end{corollary}

Finally, we obtain an expression for $\mu\in \mathcal{Q}^*(RCD)$. 
\begin{theorem}
\label{t:RCD}
Suppose that $\phi_h$ has defect measure $\mu$ and that its renormalized Cauchy data has defect measure $\mu^{RCD}\in \mathcal{Q}^*(RCD)$. Then, for $a\in C(T^*H)$, 
$$
\mu^{RCD}(a)=\int_{S^*_HM\setminus S^*H}a(\pi(\zeta))|\xi_n(\zeta)|d\nu^{\perp}(\zeta)
$$
where for $A\subset S^*_HM\setminus S^*H$,
$$
\nu^{\perp}(A):=\lim_{T\to 0^+}\frac{1}{2T}\mu\Big(\bigcup_{|t|\leq T}\varphi_t(A)\Big),
$$
and $\pi:S^*_HM\to B^*H$ denotes the orthogonal projection.
\end{theorem}

\subsection{Outline of the article} Section~\ref{LAPSECT} contains some preliminary facts about the Laplacian. This is followed by Section~\ref{RELLICHSECT} which reviews the basic calculations for Rellich's formula. Section~\ref{MDSECT} then contains the study of invariant measures and in particular, the proof of Theorem~\ref{t:invMeas}. Section~\ref{FINALSECT} contains the proof of theorem~\ref{t:RCD} and the proof of Theorem~\ref{thm1} to from Theorem~\ref{t:RCD}. Finally, Section~\ref{HYPERBOLIC} contains the proof of Theorem~\ref{t:CD} via a factorization method.

\ \\

\noindent {\sc Acknowledgements.} This work was largely written   during the period when both authors were research members at the Mathematical Sciences Research Insittute.  S.Z. would like to acknowledge support under NSF grant  DMS-1810747.

\section{\label{LAPSECT} Laplacian in Fermi normal coordinates along a hypersurface}

Let $(M, g)$ be any Riemannian manifold. We recall that in any coordinate system, $$\Delta_g = \frac{1}{\sqrt{g}} \sum_{i,j} \frac{\partial}{\partial x_i} (g^{ij} \sqrt{g} \frac{\partial}{\partial x_j}), $$
where $g = \det (g_{ij})$. Here $g_{ij} = g(\frac{\partial}{\partial x_i}, \frac{\partial}{\partial x_j})$ and $g^{ij}$ is the inverse matrix.

Let $H \subset M$ be a hypersurface, and let
  $x =(x_1,...,x_{n-1},x_n)= (x',x_n) $ be Fermi normal coordinates in a small
tubular neighbourhood $H(\epsilon)$ of $H $ defined near a point
$x_0 \in H$.  Thus, $H = \{ x_1 = 0\}$ with  coordinates $x' $ on $H$. Fermi coordinates use the charts $\exp_{x'} x_n \nu$ where $\nu$ is a choice of unit normal field. In these coordinates we can locally write
\begin{equation} \label{Hepsilon} H(\epsilon) := \{ (x',x_n) \in U \times {\mathbb R}, \, | x_{n} | < \epsilon \}. \end{equation}
Here $U \subset {\mathbb R}^{n-1}$ is a coordinate chart
containing $x_0 \in H$ and $\epsilon >0$ is small but
for the moment, fixed.  We also denote by $\xi_n, \xi'$ the symplectically dual coordinates
on $T^* H(\epsilon).$

  In Fermi  normal coordinates, the metric is given by 
  $$
  g= g_H(x, dx')+dx_n^2.
  $$
  where $g_H(x,dx')$ is a metric in $x'$ depending on $x_n$. In particular, the metric induced on $H$ is $g_H(0,x',dx')$.
  Therefore, 
  
   \begin{equation} \label{LAPF} \begin{aligned}
   -h^2 \Delta_g&= \frac{1}{\sqrt{g(x)} }(hD_{x_n}\sqrt{g(x)} hD_{x_n} +hD_{x_i}g_H^{ij}(x)\sqrt{g(x)} hD_{x_j})\\
& =    \frac{1}{\sqrt{g(x)}} hD_{x_n} \sqrt{g(x) } hD_{x_n}   + R(h, x', x_n,h D_{x'}) \\
& =    (hD_{x_n})^2    + R(h, x',x_n, h D_{x'}) +hr_{1n}(x,hD_{x_n})
\end{aligned} 
\end{equation}
    where $R$ is a second-order $h$-differential operator along $H$ with coefficients depending on $x_n$,  $r_{1n}$ is a first
order normal operator, and
    $$
    R=R_2(x',x_n,hD_{x'})+hr_1(x',x_n,hD_{x'})
    $$
where $r_1(x_n,x',hD_{x'})$ is a first order operator along $H$ with coefficients depending on $x_n$, and
   \begin{equation} \label{R2} \begin{aligned}
 R_2(x',x_n,hD_{x'}) 
   & =  R_2(x',0, h D_{x'}) + 2 x_n Q( x',x_n, h D_{x'}),  
           \end{aligned} \end{equation}
         Here, $Q( x',0, \xi')$ is the second fundamental form of $H$ and  $ R_2(x',0, hD_{x'}) = -h^2 \Delta_H$  (the
induced tangential semiclassical Laplacian on $H$).
    The semi-classical principal symbol  $ \sigma(- h ^2 \Delta_g)$ of $-h^2 \Delta_g $ is given by
\begin{equation} \label{SYMBOL}  p(x, \xi)  = \xi_n^2 + |\xi'|_{g_H}^2  . \end{equation}
We recall that the second fundamental form $II(X, Y)$ of a hypersurface $H$  is the symmetric tensor on
$TH$  defined by $II(X, Y) = \nabla^M_X Y - \nabla^H_X Y$ where
$\nabla^M$ is the covariant derivative for $(M,g)$ and $\nabla^H$ is the
covariant derivative for $(H, g|_H)$. The second fundamental form defines
a quadratic form on $T_{x'} H$ for every $x' \in H$. Hence it is given by
a quadratic polynomial $Q(0, x', \xi')$ in $\xi'$ at each $x'$. 

   The first order terms $r_1, r_{1n}$ play no role in the calculations of this paper since they only contribute to the  $O(h)$ remainder.

\subsection{\label{FSURFACE} Calculations on a surface}

Although $x_n$ denotes the $n$th Fermi coordinate in dimension $n$, we often use the same notation on a surface when we want to have uniform
notation in all dimensions, 
since $x_n$ also indicates `normal coordinate'.   In the case of a curve $H $ in a surface $M$, we use the notation $(s,y) = (x_1, x_n)$ for the  Fermi normal coordinates  along 
$H$, with
$x_n = y, x' = s$  (where $s$ is arc-length);   the symplectically dual coordinates are denoted by $(\sigma, \eta)$. The metric components 
are given  by $g^{00} = g(\frac{\partial}{\partial s}, \frac{\partial}{\partial s}),
g^{11} = g(\frac{\partial}{\partial y}, \frac{\partial}{\partial y})$ and $g^{01} = g^{10} =
 g(\frac{\partial}{\partial y}, \frac{\partial}{\partial s}) \equiv 0$.  Since the vector fields $\frac{\partial}{\partial y}$ are tangent to unit speed geodesics,
$g^{11} \equiv 1$ and  in particular, $\partial_{x_n} g^{11} = 0$. The Taylor expansion of $g^{00}(s,y)$ around   $y=0$ has the form,  $$g^{00}(s,y) =1+ 2 y \kappa_{\nu}(s) +  C_1 \tau(s) y^2 
+O(y^3), $$  where $\kappa_{\nu}(s)$ is the geodesic curvature of $H$, and where  $\tau$ is the scalar curvature of $(M,g)$.

\section{\label{RELLICHSECT} Rellich identity }

The result of Theorem \ref{thm1} is local on $H$, and with no loss of generality we may assume that
 $H$ is 
the boundary of a smooth open domain
 $M_+ \subset M$, $H = \partial M_+$, and $x_n>0$ in $M_+$.    We then use a Rellich identity to write the integral of a commutator over
$M_+$ as a sum of integrals over the boundary (of course the same
argument would apply on $M_- = M \setminus M_+$).  We follow the exposition of \cite{CTZ13} in the following and continue
to use the notation $\phi_h$ for a sequence of eigenfunctions and allow $H \subset M$ to be a hypersurface in a manifold of
any dimension.

Let $A(x, h D_x) \in \Psi_{sc}^{0}(M)$ be an order zero semiclassical pseudodifferential operator on $M$ (see \cite{Zw}). 
Also denote by $\gamma_H$ the restriction operator $\gamma_H f = f |_H$. If  $\phi_h$ is  a Laplace eigenfunction of eigenvalue $-h^{-2}$, then by   Green's formula, 
\begin{align}
\label{rellich}
 {-} \frac{i}{h} \int_{M_+}  & \left( [-h^2 \Delta_{g},  \, A(x, h D_{x}) ] \, \phi_{h}(x)  \right) \overline{\phi_{h}(x)} \,  dx \\
= & \int_H \left(\gamma_H (h D_{\nu}  \, A(x',x_n,h D_x)
 \phi_{h}\right)  \overline{\phi_{h}}_H  \,  d\sigma_{H} \notag \\
& +   \int_{H}      \left(\, A(x',x_n,h D_x) \,
  \gamma_h \phi_{h} \right) (\gamma_H \overline{hD_{\nu}
    \phi_{h}})   \, d\sigma_{H} .\notag
\end{align}
Here, $D_{x_j} = \frac{1}{i} \frac{\partial}{\partial x_j}$,   $D_{x'}=(D_{x_1},...,D_{x_{n-1}}),$ \,  $D_{\nu} = \frac{1}{i} \partial_{\nu}$ where $\partial_{\nu}$
is the interior unit normal to $M_+$.  Henceforth we often abbreviate $D_{x'}$ by $D'$. Also, $d \sigma_H$ is the surface measure on $H$.

Let $Op_h(a) = a^w(x', h D_{x'})$ be a semi-classical pseudo-differential operator on $L^2(H)$. 
We  wish to choose $A(x, h D_x)$ to so that $A(x', 0, h D_x)$ is close to $a^w(x', h D_{x'})$ and so that  $[-h^2 \Delta_g, A(x, hD_x)]$  has good positivity properties when $a^w(x', h D_{x'}) \geq 0. $ 

We let $\chi \in C^{\infty}_{0}({\mathbb
R})$ be a cutoff with $\chi(x) = 0$ for $|x| \geq 1$ and $\chi(x)
= 1 $ for $|x| \leq 1/2.$
Given $a \in S^{0,0}(T^*H \times (0,h_0]),$ we define the pseudo-differential operator $A$ on $M$ by, \begin{equation}
\label{ADEF} A(x',x_n,h D_x) = \chi ( \frac{x_n}{\epsilon}) \,h
D_{x_{n}}a^w(x', hD'). \end{equation} We now calculate the two sides of \eqref{rellich} following \cite{CTZ13}, in particular  showing that matrix elements of the commutator $[-h^2
\Delta_g, A(x, hD_x)]$  of this `extension' of $a^w(x', D_{x'})$ with $-h^2 \Delta_g$ have good positivity properties. Of course, $A$ is not truly an extension because it  is not totally characteristic, i.e. it also contains normal derivatives $hD_{x_n}$. 
\subsubsection{The right hand side}
Since $\chi(0)=1$, the second term  $$   \int_{H}   \gamma_H    \left(\, A(x',x_n,h D_x) \,
\phi_{h} \right) (\gamma_H \overline{hD_{\nu}
    \phi_{h}})   \, d\sigma_{H} = \int_H \gamma_{H}  \chi ( \frac{x_n}{\epsilon}) \,h
D_{x_{n}}a^w(x', hD') \phi_h) \overline{hD_{\nu}
    \phi_{h}} d \sigma_H$$ on the right side  of
(\ref{rellich}) is the Neumann data matrix element,
\begin{equation} \label{rellich2}
 \lll a^w(x', hD')h D_{x_n} \phi_h |_H, h D_{x_n} \phi_h |_H \rrr. \end{equation}

We now show that the first term on the right hand side   of (\ref{rellich})  is the renormalized Dirichlet data. 
Using that  $\chi'(0) = 0$ and $-h^2\Delta_g \phi_h = \phi_h$, the first term equals
\begin{equation}
 \label{rellich3}
\begin{aligned}
  \int_{H} &   \gamma_{H}  \left( h D_{n} (\chi(x_n/\epsilon)h D_n a^w(x',
  hD')\phi_{h} \right) ( \gamma_{H} \overline{ \phi_{h}})  \,
  d\sigma_{H} \\
= & \int_H \gamma_H \chi(x_n/\epsilon) a^w(x', hD') (h D_n)^2  \phi_h \gamma_H \overline{\phi_h}d\sigma_H\\
&\qquad +
\int_H \frac{h}{i \epsilon} \chi'(x_n/\epsilon)hD_n a^w(x', hD') \phi_h
\Big) \gamma_H \overline{\phi_h}d \sigma_H  \\
= & \int_H \gamma_H ( \chi(x_n/\epsilon) a^w(x', hD') (1 - R(x_n, x', hD'))
\phi_h ) \gamma_H \overline{\phi_h} d \sigma_H \\
&\qquad+ O_{\epsilon}(h)(\|\gamma_H \phi_h\|^2 +\|\gamma_H hD_{x_n}\phi_h\|^2)\\
= & \int_H  a^w(x', hD') (1 - h^2 \Delta_H)
\gamma_H \phi_h \cdot \gamma_H \overline{\phi_h}d \sigma_H\\
&\qquad+ O(h)(\|\gamma_H \phi_h\|^2 +\|\gamma_H hD_{x_n}\phi_h\|^2) .
\end{aligned}
\end{equation}
In the last line we use that $\chi(0) = 1$ and the expansions \eqref{LAPF}-\eqref{R2} together with the fact that, since $\phi_h$ is a Laplace eigenfunction 
$$\|\gamma_H r_{1}(x',0,hD_{x;})\phi_h\|_{L^2(H)}\leq C\|\gamma_H\phi_h\|_{L^2}.$$

\subsubsection{Left hand side of \eqref{rellich}} Since the semi-classical  principal symbol of $\tfrac{i}{h} [-h^2 \Delta_{g},  \, A(x, h D_{x}) ] $
equals the Poisson bracket $\{\xi_{n}^{2} +
      R_2(x_n,x',\xi') , \chi(\frac{x_n}{\epsilon})  \xi_{n}a(x', \xi')  \}$, we have

\begin{equation} \label{rellich4} \begin{aligned}  &- \frac{i}{h} \int_{M_+}   \left( [-h^2 \Delta_{g},  \, A(x, h D_{x}) ] \, \phi_{h}(x)  \right) \overline{\phi_{h}(x)} \,  dx\\ 
& \qquad\qquad=  - \left\langle  \left( \left\{
       \xi_{n}^{2} +
      R_2(x',x_n,\xi') , \, \chi(\frac{x_n}{\epsilon})  \xi_{n}a(x', \xi')  \right\}  \right)^w\phi_{h}, \,\,
  \phi_{h} \right\rangle_{L^2(M_+)}+ {\mathcal
  O}_{\epsilon}(h).
\end{aligned} \end{equation}

\subsection{\label{PBSECT} Some Poisson bracket calculations}

 Since $\xi_n^2$ is only non-trivially paired with $x_n$,
\begin{equation} \label{COMM} \begin{aligned}  &-\left\{
       \xi_{n}^{2} +
      R_2(x',x_n,\xi') , \, \chi(\tfrac{x_n}{\epsilon})  \xi_{n}a(x',
      \xi')  \right\} \\&\qquad\qquad\qquad =  - \tfrac{2}{\epsilon} \chi'(\tfrac{x_n}{\epsilon})
    \xi_{n}^{2}a(x', \xi')  + \chi(\tfrac{x_n}{\epsilon})  P_2(x',x_n,\xi', \xi_n)
    , \end{aligned}  \end{equation}
where $P_2 =- \{R_2(x',x_n, \xi'),  \xi_{n}a(x',
      \xi')   \} $.    In general dimensions, \begin{equation} \label{P2DEF} \begin{array}{l} P_2   = \frac{\partial R_2(x',x_n, \xi')}{\partial_{x_n}} a(x', \xi') -  \xi_n \{R_2(x_n, x', \xi), a(x', \xi')\}. \end{array} \end{equation}
      When restricted to $S^*H$ the second term is zero and one gets
      $$  \frac{\partial R_2(x_n, x', \xi')}{\partial_{x_n}} |_{x_n = 0}\;\;a(x', \xi') =  2Q(0, x', \xi') a(x', \xi'). $$

 In the case of a curve $H$ and in Fermi normal coordinates, 
       \begin{equation} \label{PH} P_2(s, 0, \sigma, \eta) =  2   \kappa_{\nu}(s) a(s, \sigma) - 2 \sigma  \eta   \frac{\partial ( a(s,\sigma) )}{\partial s}. \end{equation}
             The first  term vanishes when $y = 0$, since $g^{00}(s,0) = 1$, while $ \frac{\partial  g^{00}(s,y)  \sigma^2}{\partial y} a(s,\sigma) = 2 \kappa_{\nu}(s) a(s, \sigma). $
      Hence,
    The second term
      vanishes when $\eta = 0$, i.e. on $S^*H$.

\subsection{Semi-classical limit of the Rellich formula}

We consider  any sequence $\{\phi_h\}$ with a single microlocal defect  measure $\mu \in \qcal^*$.  {
It will be convenient to extend some integrals from $M_+$ to $M$. For this, we introduce a cutoff
$\tchi\in \Ci(M)$ such 
that 
\begin{equation}
\label{d:tchi}
\begin{gathered}
\chi'(x_n/\epsilon)|_{M_+} = \tchi'(x_n/\epsilon),\qquad\qquad\supp (1-\tchi)\cap x_n\geq 0 =\emptyset
\end{gathered} 
\end{equation}}
Also recall that $P_2$ is defined in \eqref{P2DEF}.
  
 \begin{proposition}\label{MDMPROP} Let $(M, g)$ be a compact Riemannian manifold and  let $H \subset M$ be a smooth, embedded, orientable hypersurface. 
  Then, \smallskip
 \begin{equation} \label{IandIIINTRO}\begin{aligned}& \left|\lim_{ h  \to 0} \left(\begin{aligned}&\lll Op_h(a)  h D_\nu \phi_h |_{H} , h D_\nu \phi_h |_H \rrr_{L^2(H)} \\&\qquad+ \lll Op_h(a)  (1
+ h^2 \Delta_H) \phi_h |_{H}, \phi_h |_{H} \rrr_{L^2(H)} \end{aligned} \right)-  I_0(a, \epsilon, \mu) \right|\\ \smallskip
&\qquad\qquad\qquad\qquad\qquad\qquad\qquad\qquad\qquad\qquad\qquad\qquad\leq  II_0(a, \epsilon, \mu). \end{aligned} \end{equation} 
where
    \begin{equation} \label{I-IIDEF} \left\{ \begin{array}{l} I_0(a, \epsilon, \mu): = {-} 2 \int_{S^*M} \frac{1}{\epsilon}
     \tchi'(\frac{x_n}{\epsilon}) \, \xi_n^2 a(x',\xi')  d\mu, \\ \\
II_0(a, \e, \mu): =\sqrt{\int \chi(\frac{x_n}{\e})^2P_2^2(x',x_n,\xi')d\mu}.
 . \end{array} \right. \end{equation}
\end{proposition}

\begin{proof}

The Rellich identity and the calculations  (\ref{rellich})-(\ref{rellich3})-\eqref{COMM} in Section \ref{PBSECT} show that, for any hypersurface  $H \subset M$, 
\begin{equation} \label{rellich6}\begin{aligned} 
 \begin{aligned} \langle &Op_h(a)  h D_\nu \phi_h |_{H} , h D_\nu \phi_h |_H \rangle_{L^2(H)} \\
  &
  + \langle Op_h(a)  (1
+ h^2 \Delta_H) \phi_h |_{H}, \phi_h |_{H} \rangle_{L^2(H)}\end{aligned}  &= I_{ h } (a, \epsilon)  + II_{ h }(a, \epsilon) +\mathcal{O}_{\epsilon}(h)
 \end{aligned}
 \end{equation}
 where
 \begin{equation}
 \begin{aligned}
 I_h&= - \left\langle  \left( \tfrac{2}{\epsilon} \chi'(\tfrac{x_n}{\epsilon})
    \xi_{n}^{2}a(x', \xi')\right)^w  \phi_{h}, \,\,
  \phi_{h} \right\rangle_{L^2(M_+)}\\
  II_h&=\left\langle \left( \chi(\tfrac{x_n}{\epsilon})  P_2(x',x_n,\xi', \xi_n)  \right)^w\phi_h,\phi_h\right\rangle_{L^2(M_+)}\end{aligned} \end{equation} 
and $P_2$ is given by \eqref{P2DEF} - \eqref{PH}.

Now, $\chi'(x_n/\epsilon)|_{M_+} = \tchi'(x_n/\epsilon)$ where $\tchi$ is as in~\eqref{d:tchi}. Therefore, 
Since  $\tchi'$ and $\chi'$ are supported inside $M_+$,
$$
   \left\langle    \left( \tfrac{1}{\epsilon}
     \chi'(\tfrac{x_n}{\epsilon}) \, \xi_n^2 a(x',\xi') \right)^w \phi_{h}, \,\,
   \phi_{h} \right\rangle_{L^2(M_+)} \\
 =  \left\langle  \left( \tfrac{1}{\epsilon}
     \tchi'(\tfrac{x_n}{\epsilon}) \, \xi_n^2 a(x',\xi') \right)^w \phi_{h}, \,\,
   \phi_{h} \right\rangle_{L^2(M)}. $$
   Sending $h\to 0$ in the right hand side yields $I_0(a,\e,\mu)$.
   
   Next, observe that by Cauchy-Schwarz,
   \begin{multline*}
   II_h(a,\e)=\left\langle \left(\chi(\tfrac{x_n}{\epsilon})  P_2(x',x_n,\xi')  \right)^w\phi_{h}, \,\,
  \phi_{h} \right\rangle_{L^2(M_+)}\\\leq \|\big(\chi(\tfrac{x_n}{\epsilon})  P_2(x',x_n,\xi')  \big)^w\phi_{h}\|_{L^2(M)}.
   \end{multline*}
   Then, 
   $$
   \lim_{h\to 0}\|\big(\chi(\tfrac{x_n}{\epsilon})  P_2(x',x_n,\xi')  \big)^w\phi_{h}\|_{L^2(M)}^2=\int \chi(\tfrac{x_n}{\e})^2P_2^2(x',x_n,\xi')d\mu.
   $$
   This completes the proof.

\end{proof}

 \section{\label{MDSECT} Decompositions of  microlocal defect measures}
 
 The next two sections are devoted to the calculation of the limits 
 $I_{0}(a, \epsilon, \mu)$  resp.  $II_{0}(a, \epsilon, \mu)$ (\eqref{I-IIDEF})  as $\epsilon \to 0$. 
We first make the decomposition  \begin{equation} \label{DECOMP} \mu = \mu |_{S^*_H M} + \mu^{\perp}, 
\;\; {\rm{ where}} \; \mu^{\perp}(S^*_H M) = 0, \end{equation}  
and where $ \mu |_{S^*_H M} = {\bf 1}_{S^*_H M} \mu$ is the restriction of $\mu$ to $S^*_H M$. Here, $S^*_H M$ is the set of unit co-vectors to $M$ with footpoint on $H$ and $S^*H \subset S^*_H M$ are those (co-)tangent to $H$. In this section, we first study the measure $\mu|_{S^*_HM}$, showing that it is supported in in $S^*H$ at points which are `nearly' totally geodesic (See Lemma~\ref{mushm=muh} and Corollary~\ref{c:squirrel}). We then calculate the limits of \eqref{I-IIDEF} in  Proposition \ref{I0PROP}.

\subsection{Disintegration of $\mu$ with respect to the geodesic flow}

We next briefly recall the theory of disintegration of measures along a fibration~\cite[Theorems 2.1.22, 4.1.17]{Du19}.

\begin{proposition}[Disintegration Theorem]
\label{t:disintegrate}
Suppose that $(Y,\mathcal{Y},\mu)$ is a probability space, $X$ is a Borel subset of a complete separable metric space, endowed with the Borel sigma algebra, and $\pi:Y\to X$ is measurable. Define $\nu:=\pi_*\mu$. Then there is a $\nu$ a.e. unique family of probability measure $\{\mu_x\}_{x\in X}$ on $Y$ such that 
\begin{itemize}
\item[(i)] for all Borel $A\subset Y$, $x\mapsto \mu_x(A)$ is measurable.
\item[(ii)] $\mu_x(Y\setminus \pi^{-1}(x))=0$
\item[(iii)] for any Borel measurable function $f: Y \to \R_+$,
\begin{equation}\label{DISINT} \int_Y f(y) d\mu(y) = \int_X \left( \int_{\pi^{-1}(x)} f(y) d\mu_x(y)\right) d\nu(x). \end{equation}
\end{itemize}
\end{proposition}

In the case of interest, we fix $\delta>0$ small and define 
$$
Y=FL_\delta(S^*_HM):= \bigcup_{|t|\leq \delta}\exp(tH_p)(S^*_HM),
$$
and the map
$$
\pi_\delta:FL_\delta(S^*_HM)\to \mathfrak{G}_\delta:=Y/\sim
$$
where $\sim$ denotes the relation of belonging to the same orbit.
Then, let $\mu^{\delta}_x$ and $\nu^{\delta}$ be the measures guaranteed by Proposition~\ref{t:disintegrate}.

Note that the quotient space $\mathfrak{G}_\delta$ is not equal to $S^*_H M$; e.g. $H = \gamma$ is a closed geodesic, then $S^*\gamma$ is a single orbit and a single point in the quotient. There is, however, a large subset of $\mathfrak{G}_\delta$ which can be easily identified with $S^*_HM$. In particular, if an orbit in $\mathfrak{G}_\delta$ intersects $S^*_HM$ only once, we may identify this orbit with its intersection with $S^*_HM$.
\begin{lemma}
\label{l:inv}
Suppose that $\mu$ is invariant under $\exp(tH_p)$ and $\rho_0\in S^*_HM$ such that there is a neighborhood, $U$ of $\rho_0$ such that for $\rho\in U$
$$
\bigcup_{|t|\leq \delta}\exp(tH_p)(\rho)\cap S^*_HM=\rho.
$$ 
Then, identifying $\rho\in S^*_HM$ with its orbit in $\mathfrak{G}_\delta$, for all $\rho\in U$, using $[-\delta,\delta]\times U\ni t\mapsto \exp(tH_p)(\rho)\in FL_\delta(S^*_HM)$ as coordinates on their image, $\mu^{\delta}_\rho(t,\zeta) = \frac{1}{2\delta}1_{[-\delta,\delta]}dt\delta_{\rho}(\zeta)$ and in particular, for $A\subset U$ Borel, 
$$
\nu^{\delta}(A)=\mu\Big(\bigcup_{|t|\leq \delta}\exp(tH_p)(A)\Big).$$
\end{lemma}
\begin{proof}
First observe that the given coordinates are valid. Next, $\mu_\rho^\delta$ is clearly supported on $\pi^{-1}(\rho)$ and is Borel measurable. Therefore, we need only check that~\eqref{DISINT} holds with $f$ supported in 
$$
FL_{\delta}(U):=\bigcup_{|t|\leq \delta}\exp(tH_p)(U).
$$ 
For that, observe that on $FL_\delta(U)$, $\mu=\nu^{\perp}(\zeta) dt$ for some $\nu^\perp$. Therefore, for all $0<T\leq \delta$,
$$
\nu^\perp(A)=\frac{1}{2T}\mu\Big(\bigcup_{|t|\leq T}\exp(tH_p)(A)\Big)=\frac{1}{2\delta}\nu^\delta(A),
$$
and the lemma follows.
\end{proof}
For future use, we define 
\begin{equation}
\label{e:disintegrate}
\nu^\perp(A):=\frac{1}{2\delta}\nu^\delta(A)=\lim_{T\to 0^+}\frac{1}{2T}\mu\Big(\bigcup_{|t|\leq T}\exp(tH_p)(A)\Big).
\end{equation}

\subsection{Disintegration of $\mu$ with respect to the normal fibration}
It is also possible to disintegrate $\mu$ with respect to the Fermi   normal fibration over $H$. Let
$$
S^*H(\epsilon):=\{ (x',x_n,\xi)\in S^*M\mid |x_n|<\e\}.
$$Let $H_{\xi_n} =\frac{\partial}{\partial x_n}$ be the Hamilton vector field of $\xi_n$ on 
$|x_n|<\epsilon$. Its Hamilton flow is given by $\psi_t(x', x_n, \xi', \xi_n) = 
(x', x_n +t, \xi', \xi_n)$. In these coordinates $S^*_H M$ is defined by $x_n=0$ and the integral
curves of $\psi_t$ define a fibration over $S^*_H M$. Given $(x, \xi) \in S^*_x M$ with $x \in \tcal_{\delta}(H)$, parallel translate $\xi$ along the normal geodesic from $x$ to $H$. 
Denote the result by $P_x^{x'} \xi$. 
Define \begin{equation} \label{pidef1}\left\{\begin{array}{l}  \pi_{\delta}: S^* \tcal_{\delta}(H) \to S^*_H M, \;\;\pi_{\delta} (x', x_n, \xi)  := (x', P_x^{x'} \xi), \\ \\  
\mu^{\delta}_H = \pi_{\delta *}  d\mu |_{S^* H(\epsilon)}. \end{array} \right. \end{equation}
For $\delta$ very small, this map is well-approximated by the map,
 \begin{equation} \label{pidef2}  \begin{array}{l} \wt \pi_{\delta}: S^* H(\epsilon) \to T^*_H M, \;\; \wt \pi_{\delta} (x', x_n, \xi', \xi_n)  := (x', 0, \xi', \xi_n), \end{array}.\end{equation}
which however is not normalized so that the image lies in $S^*M$.

Applying Proposition~\ref{t:disintegrate} there exist finite fiber  measures
 $d\mu^{\epsilon}_{\rho}$  on the fiber of \eqref{pidef1} over $\rho \in S^*_H M$ such that
\begin{equation} \label{DISF}
\int_{S^* H(\epsilon)} f d\mu = \int_{S^*_H M} \left(\int_{\pi_{\e}^{-1}(\rho)} f  d\mu_{\rho}^{\epsilon} \right) d\mu^{\epsilon}_H. 
\end{equation}

The principal defect of this disintegration is that the Fermi normal fibration is not invariant under $\phi_t$, and thus the disintegrated measures are more difficult to compute. For instance, if $\mu = \delta_{\gamma}$ is a periodic orbit 
measure, the fiber measure $d\mu_{\zeta}$ is singular with respect to Lebesgue measure along the Fermi normal fibers,
and does not possess a derivative at $x_n = 0$. 

This type of fibration could be used in Section 5 for the proof of Lemma~\ref{l:calcII}, but we find it simpler to use the geodesic fibration.

\subsection{The behavior of $\mu |_{S^*_H M} $ and $ \mu |_{S^*_HM}$}

The purpose of this section is to prove Theorem~\ref{t:invMeas}


\begin{proof}
The proof consists of several Lemmas which yield stronger versions of the conclusions.

\begin{lemma}

\label{mushm=muh}
The measure $\mu|_{S^*_HM}$ satisfies
$$
\mu|_{S^*_HM}( \mathcal{T}_+)=0
$$
where
$$
\mathcal{T}_+:=\{\rho\in S^*_HM\mid T_{S^*H}(\rho)>0\},\qquad T_{S^*H}(\rho):=\inf\{ t>0\mid \phi_t(\rho)\in S^*H\}.
$$
Moreover, 
$$
\mu|_{S^*_HM}(A)\leq  \liminf_{T\to 0}  \frac{|\{t\in[\min(0,T),\max(0,T)]\mid \varphi_{-t}(A\cap S^*H)\cap A\cap S^*H\}|}{|T|}.
$$
\end{lemma}

   \begin{proof}

  Let $A\subset S^*_HM$ Borel measurable. Then, for any $T>0 $
  \begin{equation}
  \label{e:squirrel}
  \begin{aligned}
  \mu|_{S^*_HM}(A)=\mu(A)&=\frac{1}{T}\int_0^{T}\mu(1_{\varphi_{-t}(A)})dt=\frac{1}{T}\mu\Big(\int_0^T 1_{\varphi_{-t}(A)}(\rho)dt\Big) d\mu(\rho)\\
  &\leq \frac{|\{t\in[0,T]\mid \varphi_t(A)\cap A\neq \emptyset\}|}{T}\\
  &\leq  \frac{|\{ t\in[0,T]\mid \varphi_t(A)\cap S^*_HM\neq \emptyset\}|}{|T|}
  \end{aligned}
  \end{equation}

Now, define
 $$(S_H^*M)_{\delta} = \{\zeta \in S^*_H M: T_{S^*_HM}(\zeta) > 2 \delta\},$$
where $$T_{S^*_HM}(\zeta):= \inf_{t > 0} \{\phi_t(\zeta) \in S^*_H M \}. $$
Then, since $T_{S^*_HM}$ is lower semincontinuous, $(S_H^M)_{\delta}$ is open and hence measurable. Therefore by~\eqref{e:squirrel}
$$
\mu|_{S^*_HM}( (S_H^*M)_{\delta})=0
$$
and hence 
$$
\mu|_{S^*_HM}( \{\zeta \mid T_{S^*_HM}(\zeta)>0\})=0.
$$

Note that $S^*_HM\setminus S^*H\subset  \{\zeta \mid T_{S^*_HM}(\zeta)>0\}$ and hence, 
$$
\mu|_{S^*_HM}=\mu|_{S^*H}.
$$

Therefore, arguing as in~\eqref{e:squirrel}, 
\begin{equation}
\label{e:squirrel2}
  \mu|_{S^*_HM}(A)\leq \liminf_{T\to 0}  \frac{|\{t\in[\min(0,T),\max(0,T)]\mid \varphi_{-t}(A\cap S^*H)\cap A\cap S^*H\}|}{|T|}
\end{equation}
Now, with $T_{S^*H}$ as above, define 
$$
(S^*H)_\delta:=\{\zeta\in S_H^*M\mid T_{S^*H}(\zeta)>\delta\}.
$$
Then~\eqref{e:squirrel2} implies $\mu|_{S_H^*M}( T_{S^*H}>0)=0$.

   \end{proof}

 \begin{corollary}
 \label{c:squirrel}
Define
$$
\mathcal{G}^k:=\{\rho\in S^*_HM\mid [H_p^kx_n](\rho)\neq 0,\, [H_p^jx_n](\rho)=0,j<k\}.
$$
Then, 
$$
\mu|_{S^*_HM}\Big(\bigcup_{k=0}^\infty\mathcal{G}^k\Big)=0.
$$
In particular, 
$$
\mu|_{S^*_HM}=\mu|_{S^*H}
$$
and
$$
\mu|_{S^*_HM}\big(\{ (0,x',\xi')\mid Q(0,x',\xi')\neq 0\}\big)=0.
$$
\end{corollary}
\begin{proof}
Observe that if $\rho\in \mathcal{G}^k$, then, 
$$
|x_n(\phi_t(\rho))|\geq ct^k+O(t^{k+1})
$$
and in particular, $T_{S^*H}(\rho)>0$.  Therefore, $\mathcal{G}^k\subset \mathcal{T}_+$ and the claim follows.
\end{proof}

 This concludes the proof of Theorem~\ref{t:invMeas}.
 
\end{proof}

\subsection{A conjecture} Theorem \ref{t:invMeas} and accompanying Lemmas leave open some purely dynamical
questions concerning the the restriction  $\mu |_{S^*_H M } = \mu |_{S^*H}$ of an invariant measure. We state a conjecture which
we hope to explore  in the future.

We denote by $\gamma$ a geodesic of $(M,g)$ and also (by abuse of notation)  the corresponding orbit of the geodesic flow in $S^*M$.
When (and only when) $\gamma$ is a periodic geodesic, we  denote by $\delta_{\gamma}$ the normalized periodic orbit measure $\delta_{\gamma}(f) = \frac{1}{L_{\gamma}} \int_{\gamma} f ds$ where $L_{\gamma}$ is the length of $\gamma$. 

\begin{conjecture} Suppose that $\mu$ is an invariant probability measure for the geodesic flow of a compact Riemannian manifold. Suppose
that $H \subset M$ is a smooth hypersurface and that $\mu (S^*H) > 0$. Then $\mu|_{S^*H}$ is supported on a union of periodic geodesics $\gamma$ such
that $\gamma \cap S^*H$ has positive arc-length measure and $\mu |_{S^* H \cap \gamma} \ll \delta_{\gamma}$ (i.e. is absolutely continuous).
\end{conjecture}

The conjecture is simplest in dimension two, when $\dim H = \dim \gamma$. In that case one consequence of the conjecture is that if $\mu(S^*H)>0$, then $H$ has positive measure intersection with a periodic geodesic. In the case where $(M,g)$ is of negative curvature, each invariant measure
is an orbital averaging measure over the orbit through a quasi-regular point. This orbit may touch $S^*H$ repeatedly,  in a quasi-periodic fashion, or it may spiral
in to $S^*H$ over a part of the orbit. If the orbital average charges $S^*H$, we conjecture that it must contain a periodic orbit measure as an ergodic component.

   \section{\label{FINALSECT} Rellich proof of Theorem \ref{thm1} }
 
 We can now state the main ingredient in the proof of Theorem \ref{t:RCD}. Once that theorem is proved, we will finish the section by proving Theorem~\ref{thm1}.
\begin{proposition} \label{I0PROP} Let $H \subset M$ be a hypersurface  in a Riemannian manifold and $I_0(a, \delta, \mu)$ 
and $II_0(a, \delta, \mu)$ be as  in Proposition \ref{MDMPROP}. Then
$$\left\{ \begin{aligned} (i) & 
\liminf_{\delta \downarrow 0} I_0(a, \delta, \mu) =   \int_{S^*_H M} a(\zeta)  |\xi_n(\zeta)| 
 d \nu^{\perp}(\zeta) ;\\ (ii) &
\lim_{\delta \downarrow 0} II_0(a, \delta, \mu) = 0
.  \end{aligned} \right.$$
where $\nu^{\perp}$ is defined in \eqref{e:disintegrate}.  



\end{proposition}




\subsection{Proof of Proposition \ref{I0PROP}(ii)}

By \eqref{I-IIDEF}, Proposition  \ref{I0PROP}(ii) asserts the following:

\begin{lemma} \label{II0PROP}  Let $H \subset M$ be a hypersurface. For any fixed $\epsilon > 0$,
 $$\begin{array}{l}  \lim_{\delta \to 0} 
 
	 \int_{S^*M} \chi^2 (\frac{x_n}{\delta})  P^2_2 (x', \xi', x_n) d\mu
	 = 0.\end{array}$$
 \end{lemma}
 
 \begin{proof}

Note that by the dominated convergence theorem,
 $$
  \lim_{\delta \to 0}  
 \int_{S^*H} \chi^2 (\frac{x_n}{\delta})  P^2_2 (x', \xi', x_n) d\mu    =\int  Q^2(0, x', \xi') a^2(x', \xi') d\mu |_{S^*H} =0 
 $$
 where the last equality follows from Corollary~\ref{c:squirrel}.
  \end{proof}

 The following lemma completes the proof of Proposition \ref{I0PROP}.

\begin{lemma}
\label{l:calcII}
We have 
$$
\lim_{\e\to 0}I_0(a,\e,\mu)=\int_{S^*_HM\setminus S^*H}|\xi_n|a(\pi(q))d\nu^{\perp}(q).
$$
where $\nu^{\perp}$ is defined in \eqref{e:disintegrate}. 
\end{lemma}
\begin{proof}
Let $\chi_1\in C_c^\infty(-2,2)$ with $\chi_1\equiv 1$ on $[-1,1]$. Then, observe that by the dominated convergence theorem, for any $\delta >0$, 
$$
\lim_{\e \to 0}\int_{S^*M}\frac{1}{\e}\tilde{\chi}'\Big(\frac{x_n}{\e}\Big)\xi_n^2\chi_1(\delta \e^{-\frac{1}{2}}\xi_n)a(x',\xi')d\mu=0.
$$
since the integrand is bounded by $\delta^{-2}$ on $S^*M$. {Next, observe that  for $(x(t),\xi(t))=\exp(tH_p)(x'_0,0,\xi'_0,\xi_n)$, 
\begin{equation}
\begin{aligned}
\dot{x}_n(t)=2\xi_n(t),\qquad |\xi_n(t)-\xi_n(0)|\leq C_1|t|,\qquad C_1:=\sup_{S^*M}|H_p\xi_n|.
\end{aligned}
\end{equation}
Therefore, the map
$$
\Psi(t,\zeta)\in \{(t,\zeta)\in \mathbb{R}\times S^*_HM\mid |t|<C_1^{-1}|\xi_n(\zeta)|\}\mapsto \exp(tH_p)(\zeta)\in S^*M
$$ 
is one to one. Suppose that $\zeta_0\in S^*M$ with $0<|x_n(\zeta_0)|\leq \frac{1}{3C_1^{-1}}|\xi_n(\zeta_0)|^2$. We will show that $\zeta_0$ lies in the image of $\Psi$. Since the arguments in other cases are the same, we assume $x_n(\zeta_0),\,\xi_n(\zeta_0)>0$. Then with $(x(t),\xi(t))=\exp(tH_p(\zeta_0))$,
$$
x_n(t)=x_n(\zeta_0)+2\int_0^t \xi_n(s)ds,\qquad |\xi_n(s)-\xi_n(\zeta_0)|\leq C_1|s|.
$$ 
Therefore, for $-\frac{1}{2C_1}|\xi_n(\zeta_0)|\leq t\leq 0$,  
$$
x_n(t)\leq x_n(\zeta_0)+\xi_n(\zeta_0)t
$$
In particular, 
$$
x_n(-\frac {1}{3C_1}|\xi_n(\zeta_0)||)\leq x_n(\zeta_0)-\frac{1}{3C_1}|\xi_n(\zeta_0)|^2\leq 0,
$$
and there is $t\in [-\frac {1}{3C_1}|\xi_n(\zeta_0)|,0]$ such that $x_n(t)=0$ and $\xi_n(t)\geq \xi_n(\zeta_0)/2$. Therefore, $\zeta_0$ lies in the image of $\Psi$. 
In particular, $(t,\zeta)\mapsto \Psi(t,\zeta)$ can be used as coordinates on
$$
\{|x_n|\leq \frac{1}{3C_1}|\xi_n|^2\}.
$$ }

Choosing $\delta>0$ small enough these coordinates are valid on 
$$
\supp \frac{1}{\e}\tilde{\chi}'\Big(\frac{x_n}{\e}\Big)\xi_n^2(1-\chi_1(\delta \e^{-\frac{1}{2}}\xi_n)a(x',\xi').
$$
Next, recall that by Lemma~\ref{l:inv} in these coordinates $\mu=dt d\nu^\perp(\zeta)$.
\begin{align*}
&\int_{S^*M}\frac{1}{\e}\tilde{\chi}'\Big(\frac{x_n}{\e}\Big)\xi_n^2[1-\chi_1(\delta \e^{-\frac{1}{2}}\xi_n)]a(x',\xi')d\mu\\
&=\int_{S^*_HM}\int_{\mathbb{R}}\frac{1}{\e}\tilde{\chi}'\Big(\frac{x_n(t,\zeta)}{\e}\Big)[\xi_n(t,\zeta)]^2[1-\chi_1(\delta \e^{-\frac{1}{2}}\xi_n(t,\zeta)]a((x',\xi')(t,\zeta))dtd\nu^{\perp}(\zeta)
\end{align*}
Now, since on the support of the integrant $c|\xi_n(0)|\leq |\xi_n(t)|\leq C|\xi_n(0)|$, and $\dot{x_n}(t)=2\xi_n(t)$, we can change variables $w=\e^{-1}x_n(t,q)$ to obtain
\begin{align*}
&\int_{S^*M}\tilde{\chi}'\Big(\frac{x_n}{\e}\Big)\xi_n^2[1-\chi_1(\delta \e^{-\frac{1}{2}}\xi_n)]a(x',\xi')d\mu\\
&=\frac{1}{2}\int_{S^*_HM}\int_{\mathbb{R}}\tilde{\chi}'(w)|\xi_n(\e w,\zeta)|[1-\chi_1(\delta \e^{-\frac{1}{2}}\xi_n(\e w,\zeta)]a((x',\xi')(\e w,\zeta))dwd\nu^{\perp}(\zeta).
\end{align*}
Then, sending $\e\to 0$ and applying the dominated convergence theorem, we obtain
\begin{align*}
\lim_{\e\to 0^+}&\int_{S^*M}\tilde{\chi}'\Big(\frac{x_n}{\e}\Big)\xi_n^2[1-\chi_1(\delta \e^{-\frac{1}{2}}\xi_n)]a(x',\xi')d\mu\\
&=\frac{1}{2}\int_{S^*_HM}\int_{\mathbb{R}}\tilde{\chi}'(w)|\xi_n(\zeta)|1_{|\xi_n|>0}a(\pi(\zeta))dwd\nu^{\perp}(\zeta)\\
&=-\frac{1}{2} \int_{S^*_HM}\int_{\mathbb{R}}|\xi_n(\zeta)|1_{|\xi_n|>0}a(\pi(\zeta))d\nu^{\perp}(\zeta)
\end{align*}
where $\pi:S^*_HM\to B^*H$ denotes the orthogonal projection map.
\end{proof}

\begin{proof}[Completion of the proof of Theorems~\ref{t:RCD}]
Observe that by Proposition~\ref{MDMPROP}, 
$$
\begin{aligned} \left|\lim_{ h  \to 0} \left(\begin{aligned}&\lll Op_h(a)  h D_\nu \phi_h |_{H} , h D_\nu \phi_h |_H \rrr_{L^2(H)} \\&\qquad+ \lll Op_h(a)  (1
+ h^2 \Delta_H) \phi_h |_{H}, \phi_h |_{H} \rrr_{L^2(H)} \end{aligned} \right)-  I_0(a, \epsilon, \mu) \right|\\
&\leq  II_0(a, \epsilon, \mu). \end{aligned}
$$
Therefore, since by Lemma~\ref{II0PROP}, $II_0(q,\epsilon,\mu)\underset{\e\to 0}{\longrightarrow} 0$, 
$$
\lim_{\e\to 0}\lim_{ h  \to 0} \left(\begin{aligned}&\lll Op_h(a)  h D_\nu \phi_h |_{H} , h D_\nu \phi_h |_H \rrr_{L^2(H)} \\&\qquad+ \lll Op_h(a)  (1
+ h^2 \Delta_H) \phi_h |_{H}, \phi_h |_{H} \rrr_{L^2(H)} \end{aligned} \right)-  I_0(a, \epsilon, \mu)=0
$$
and, since the term in parentheses is independent of $\e$, 
\begin{align*}
&\lim_{ h  \to 0} \left(\begin{aligned}&\lll Op_h(a)  h D_\nu \phi_h |_{H} , h D_\nu \phi_h |_H \rrr_{L^2(H)} \\&\qquad+ \lll Op_h(a)  (1
+ h^2 \Delta_H) \phi_h |_{H}, \phi_h |_{H} \rrr_{L^2(H)} \end{aligned} \right)\\&=\lim_{\e\to 0} I_0(a, \epsilon, \mu)\\
&=\int_{S^*_HM\setminus S^*H}|\xi_n|a(\pi(\zeta))d\nu^{\perp}(\zeta).
\end{align*}
where the last equality follows from Lemma~\ref{l:calcII}. This completes the proof of Theorem~\ref{t:RCD}.
\end{proof}
    
 \subsection{\label{COMPLETIONSECT} Completion of the proof of Theorem \ref{thm1}}

 To complete the proof of Theorem \ref{thm1} we need to combine Theorem~\ref{t:RCD}  the Dyatlov-Jin(-Nonnenmacher) theorem. Suppose Theorem~\ref{thm1} is false. Then, there exists $a\in C_c^\infty(T^*H)$ with $a\geq 0$ and $\supp a\cap B^*H\neq \emptyset$, $h_j\to 0$ such that $\phi_{h_j}$ is a Laplace eigenfunction with eigenvalue $h_j^{-2}$ and 
 $$
 \lim_{j\to \infty}\langle Op_{h_j}(a)(1+h_j^2\Delta_H)\phi_{h_j},\phi_{h_j}\rangle_{L^2(H)}+\langle Op_{h_j}h_jD_\nu\phi_{h_j},h_jD_\nu\phi_{h_j}\rangle_{L^2(H)}=0. 
 $$
 Now, we can extract a subsequence such that $u$ has defect measure $\mu$ and its renormalized Cauchy data has defect measure $\mu^{RCD}$. Then, by Theorem~\ref{t:RCD}
  \begin{align*}
 \lim_{j\to \infty}&\langle Op_{h_j}(a)(1+h_j^2\Delta_H)\phi_{h_j},\phi_{h_j}\rangle_{L^2(H)}+\langle Op_{h_j}h_jD_\nu\phi_{h_j},h_jD_\nu\phi_{h_j}\rangle_{L^2(H)}\\
 &=\int ad\mu^{RCD}=\int_{S^*_HM\setminus S^*H} a(\pi(\zeta))|\xi_n(\zeta)|d\nu^{\perp}(\zeta).
 \end{align*}
 
The proof of Theoerm~\ref{thm1} will be completed by the following lemma. 
 \begin{lem} \label{DJI} Let  $(M,g)$ be a negatively curved surface, and let $H \subset M$ be a smooth curve. Then for any $a \in C(B^* H)$ satisfying
 $a \geq 0$ and $a \not= 0$,  there exists $C_{a, H} > 0$ so that
  $$\int_{S^*_HM\setminus S^*H}|\xi_n|a(\pi(\zeta))d\nu^{\perp}(\zeta)>C_{a,H}> 0. $$
  \end{lem}
  
  \begin{proof}

 By the Dyatlov-Jin(-Nonnenmacher)theorem, the support of $\mu$ is $S^*M$. Hence, $\mu \not= \mu |_{S^*H}$, and the support
 of $\mu^{\perp}$ is also $S^*M$. It then follows from the invariance of $\mu$ under the geodesic flow that the support of $d\nu^{\perp}$ is equal to $S^*_H M$. Since $a$ is continuous, there is $q\in \supp a$ with   $|\xi_n(\zeta)| > 0$.  In particular, there is an open neighborhood $U$ of $\zeta$ such that $a(\pi(\zeta))|\xi_n(\zeta)|>c>0$ In particular, since $a\geq 0$,
$$\int_{S^*_HM\setminus S^*H}|\xi_n|a(\pi(\zeta))d\nu^{\perp}(\zeta)\geq c\nu^\perp (U)C_{a,H}> 0 $$ .
   
 \end{proof}

\section{Proof of Theorem \ref{t:CD} via hyperbolic equations}
\label{HYPERBOLIC}

In this section, we use hyperbolic equations to prove Theorem \ref{t:CD}. The idea is that $H$ is a Cauchy surface for a hyperbolic problem. In particular, when geodesics intersect $H$ transversally, we can think of $H$ as a Cauchy surface for the problem
$$
((hD_{x_n})^2-R(x,hD_{x'}))u=0,
$$
where $H=\{x_n=0\}$. Therefore, microlocally in this region,  $(u|_{H}, hD_\nu u|_{H})\mapsto u$ is a continuous map.

We start by factoring the operator $-h^2\Delta_g-1$ in the hyperbolic region. This lemma is a semiclassical version of~\cite[Lemma 23.2.8]{HoIII}.
\begin{lem}
\label{l:factor}
For all $\epsilon>0$ and $\delta>0$ small, there are 
$$\Lambda_{\pm}=\Lambda_{\pm}\in C^\infty((-\delta,\delta); \Psi^{\comp}(\mathbb{R}^{n-1})),$$ 
$\tilde{\Lambda}_{\pm}=\tilde{\Lambda}_{\pm}(x,hD_{x'})$ with
$$
\sigma(\Lambda)=\sigma(\tilde{\Lambda})=\sqrt{1-r(x,\xi')},\qquad r(x,\xi')<1-\epsilon^2
$$
such that for all $b\in C^\infty((-\delta,\delta); S^{\comp}(T^*\mathbb{R}^{n-1}))$ with $\supp b\subset \{r(x,\xi')<1-\epsilon\}$, 
\begin{align*}
b(x,hD_{x'})(-h^2\Delta_g-1)&=b(x,hD_{x'})(hD_{x_n}-\Lambda_-)(hD_{x_n}+\Lambda_+) +O(h^\infty)_{L^2\to L^2}\\
&=b(x,hD_{x'})(hD_{x_n}+\tilde{\Lambda}_+)(hD_{x_n}-\tilde{\Lambda}_-) +O(h^\infty)_{L^2\to L^2}\\
\end{align*}
\end{lem}
\begin{proof}
Fix $\chi=\chi(x,\xi')\in C_c^\infty(\mathbb{R}^{2n-1})$ with 
$$\chi \equiv 1\text{ on }\{r(x,\xi')<1-\e^2\},\qquad \supp \chi\subset \{r(x,\xi')<1-\frac{\e^2}{2}\}
$$ 
and set
$$
\lambda_0= \chi\sqrt{1-r(x,\xi')},\qquad \Lambda_0=\lambda_0(x,hD_{x'}).
$$
Recall that 
$$
-h^2\Delta_g-1=Op_h(|\xi_n|^2+r(x,\xi')-1)+h(a(x)hD_{x_n}+e(x,hD_{x'})).
$$
Then, 
\begin{align*}
&b(x,hD_{x'})(hD_{x_n}-\Lambda_0)(hD_{x_n}+\Lambda_0)\\
&=b(x,hD_{x'})(hD_{x_n}^2-\Lambda_0^2+[hD_{x_n},\Lambda_0])\\
&=b(x,hD_{x'})(-h^2\Delta_g-1-ha(x)hD_{x_n}-hr_0(x,hD_{x'}))+O(h^\infty)_{L^2\to L^2}.
\end{align*}
To obtain a finer factorization for $(-h^2\Delta_g-1)$, we put
$$
\lambda_1^-=a(x)+\frac{r_0\chi^2}{2\lambda_0},\qquad \lambda_1^+=\frac{r_0\chi^2}{2\lambda_0}
$$
and write
$$
\Lambda_1^-=\Lambda_0+h\lambda_1^-(x,hD_{x'}),\qquad \Lambda_1^+=\lambda_0+h\lambda_1^+(x,hD_{x'}).
$$
Then, we have
$$
b(x,hD_{x'})(hD_{x_n}-\Lambda_1^-)(hD_{x_n}+\Lambda_1^+)=b(x,hD_{x'})(-h^2\Delta_g-1-h^2r_1(x,hD_{x'}))+O(h^\infty)_{L^2\to L^2}.
$$
Define $r_j(x,\xi')$, $j\geq 1$ iteratively by
$$
b(x,hD_{x'})(hD_{x_n}-\Lambda_j^-)(hD_{x_n}+\Lambda_j^+)=b(x,hD_{x'})(-h^2\Delta_g-1-h^{j+1}r_j(x,hD_{x'}))+O(h^\infty)_{L^2\to L^2}.
$$
Then, for $j\geq 2$, define $\Lambda_{j}^{\pm}$ by
$$
\Lambda_j^{\pm}=\Lambda_{j-1}^{\pm}+h^j\lambda_j(x,hD_{x'}),\qquad \lambda_j(x,hD_{x'})=\frac{r_{j-1}\chi^2}{2\lambda_0}.
$$
Letting $\Lambda_{\pm}\sim \sum_{j=0}^\infty h^j\lambda_{j}^{\pm}(x,hD_{x'}),$ the claim is proved for the first factorization.

The proof is nearly identical for the other factorization.

\end{proof}

Next, we use the factorization from Lemma~\ref{l:factor} to produce propagation estimates for the Cauchy problem posed on $\{x_n=0\}$.
\begin{lem}
\label{l:propagate}
Let $b_0\in C_c^\infty(T^*H)$ with $\supp b_0\subset B^*H:=\{(x',\xi')\mid r(0,x',\xi')<1\}$. Then there is $\delta>0$ such that if $a\in C_c^\infty(T^*M)$ with 
$$
\supp a\cap S^*M\subset \bigcup_{|t|<\delta} \varphi_t(\pi^{-1}(\{|b_0|>0\}))
$$
where $\pi:S_H^*M\to T^*H$ denotes the projection and $\varphi_t:=\exp(tH_{|\xi|_g^2})$, we have the estimate
\begin{multline*}
\|Op_h(a)u\|_{L^2(M)}\leq C\|Op_h(b_0)u|_H\|_{L^2(H)}+C\|Op_h(b_0)h\partial_\nu u|_{H}\|_{L^2(H)}\\
+Ch^{-1}\|(-h^2\Delta_g-1)u\|_{L^2})+O(h^\infty)\|u\|_{L^2}.
\end{multline*}
\end{lem}
\begin{proof}
Fix $b_0$ as above and let $\e>0$ such that $\supp b_0\subset \{r(x,\xi')<1-\e\}$.  Next, let $\Lambda_{\pm}$, $\tilde{\Lambda}_{\pm}$ as in~Lemma~\ref{l:factor} and $\lambda=\sqrt{1-r(x,\xi')}$. Finally, let $\tilde{b}_0\in C_c^\infty(T^*H)$ with $\supp \tilde{b}\subset \supp b_0$, and $c>0$ such that 
\begin{equation}
\label{e:suppa}
\supp a\cap S^*M\subset \bigcup_{|t|<\delta} \varphi_t(\pi^{-1}(\{|\tilde{b}|>c/2\})).
\end{equation}

We start by defining $b^-\in C^\infty((-3\delta_0,3\delta_0); S^{\comp}(T^*\mathbb{R}^{n-1})$ such that 
\begin{equation}
\label{e:commuting}
\operatorname{WF_h}([Op_h(b^-), D_{x_n}-\Lambda_-]\cap \{|x_n|<2\delta_0\}=\emptyset.
\end{equation}
and $b^-(0,x',\xi')=\tilde{b}$.
To do this, define $\tilde{b}_0=\tilde{b}_0(x,\xi')$ by
\begin{equation}
\label{e:transport}
\tilde{b}_0(0,x',\xi')=\tilde{b},\qquad (\partial_{x_n}-H_\lambda) \tilde{b}_0=0,
\end{equation}
Next, define iteratively for $j\geq 1$,
$$
h^{j}Op_h(e_{j-1})=ih^{-1}\Big[hD_{x_n}-\Lambda_-, Op_h\Big(\sum_{k=0}^{j-1}h^k\tilde{b}_k\Big)\Big],\qquad \begin{aligned}(\partial_{x_n}-\lambda)\tilde{b}_j&=e_j,\\
\tilde{b}_j(0,x',\xi')&=0.\end{aligned}
$$
Then, putting $b^-\sim \sum_j h^j \tilde{b}_j$, we have~\eqref{e:commuting}.

Note that there is $\delta_0>0$ depending only on $\e>0$ such that a solution to~\eqref{e:transport} exists for $|x_n|<\delta_0$ and, 
$$
\supp b^-\cap \{|x_n|<3\delta_0\}\subset \{r(x,\xi')<1-\e^2\}.
$$

By standard energy estimates (see e.g ~\cite[Lemma 23.1.1]{HoIII})
\begin{multline*}
\|b^-(x,hD_{x'})(hD_{x_n}+\Lambda_+)u\|_{L^2}\\
\leq C(\|Op_h(\tilde{b})(hD_{x_n}+\Lambda_+)\|_{L^2(H)}+h^{-1}\|(hD_{x_n}-\Lambda_-)b^-(x,hD_{x'})(hD_{x_n}+\Lambda_+)u\|_{L^2(|x_n|<\delta_0)}.
\end{multline*}
Next, observe that by~\eqref{e:commuting}, for $\chi\in C_c^\infty((-2\delta_0,2\delta_0))$
$$
\chi(x_n)[hD_{x_n}-\Lambda, b^-(x,hD_{x'})]=O(h^\infty)_{L^2\to L^2}
$$
and hence
\begin{align*}
&\|b^-(x,hD_{x'})(hD_{x_n}+\Lambda_+)u\|_{L^2}\\
&\leq C(\|Op_h(\tilde{b})(hD_{x_n}+\Lambda_+)\|_{L^2(H)}+h^{-1}\|b^-(x,hD_{x'})(hD_{x_n}-\Lambda_-)(hD_{x_n}+\Lambda_-)u\|_{L^2}\\
&\qquad +O(h^\infty)\|(hD_{x_n}+\Lambda_+)u\|_{L^2}\\
&= C(\|Op_h(\tilde{b})(hD_{x_n}+\Lambda_+)\|_{L^2(H)}+h^{-1}\|(-h^2\Delta_g-1)u\|_{L^2}\\
&\qquad +O(h^\infty)\|(hD_{x_n}+\Lambda_+)u\|_{L^2}+O(h^\infty)\|u\|_{L^2}\\
\end{align*}
Next, by the elliptic parametrix construction
$$
\|(hD_{x_n}-\Lambda_-)u\|_{L^2}\leq C\|(-h^2\Delta_g+1)u\|_{L^2}\leq C(\|(-h^2\Delta_g-1)u\|_{L^2}+\|u\|_{L^2})
$$
Therefore, 
\begin{multline}
\label{e:squirrel1}
\|b^-(x,hD_{x'})(hD_{x_n}+\Lambda_+)u\|_{L^2}\\
\leq C(\|Op_h(\tilde{b})(hD_{x_n}+\Lambda_+)u\|_{L^2(H)}+h^{-1}\|(-h^2\Delta_g-1)u\|_{L^2}+O(h^\infty)\|u\|_{L^2}).
\end{multline}

Next, we construct $b^+\in C^\infty((-3\delta_0,\delta_0); S^{\comp}T^*\mathbb{R}^{n-1})$, such that 
defining $b^+$ such that  
\begin{equation}
\label{e:commuting2}
\operatorname{WF_h}([Op_h(b^+), D_{x_n}+\tilde{\Lambda}_+]\cap \{|x_n|<2\delta_0\}=\emptyset.
\end{equation}
In particular, we start with $\tilde{b}_0=\tilde{b}_0(x,\xi')$ such that 
\begin{equation}
\label{e:transport2}
\tilde{b}_0(0,x',\xi')=\tilde{b},\qquad (\partial_{x_n}+H_\lambda) \tilde{b}_0=0
\end{equation}
and proceed as in the construction of $b^-$.

We then obtain the estimate
\begin{multline}
\label{e:octopus1}
\|b^+(x,hD_{x'})(hD_{x_n}-\tilde{\Lambda}_-)u\|_{L^2}\\
\leq C(\|Op_h(\tilde{b})(hD_{x_n}-\Lambda_-)u\|_{L^2(H)}+h^{-1}\|(-h^2\Delta_g-1)u\|_{L^2} +O(h^\infty)\|u\|_{L^2}).
\end{multline}

Next, observe that on $\{\mp\xi_n>0\}\cap S^*M$, 
$$
(\partial_{x_n}\pm H_{\lambda})= (\xi_n\mp \lambda)^{-1} H_{|\xi|_g^2-1}.
$$
Therefore, by~\eqref{e:transport} and~\eqref{e:transport2}, $\sigma(b^{\pm})$ is locally invariant under the geodesic flow on $S^*M\cap \{\mp \xi_n>0\}$. In particular, there is $\delta_1>0$ such that 
$$
\{|b^{\pm}|>c>0\}\cap S^*M\supset \bigcup_{|t|<\delta_1} \varphi_t( \{(x,\xi)\in S_H^*M\mid \mp \xi_n>0,\, |\tilde{b}(x,\xi')|>c/2>0\}.
$$
Therefore, there is $c>0$ such that 
$$
[b^+(\xi_n-\lambda)]^2+[b^-(\xi_n+\lambda)]^2>c>0\qquad\text{ on }  \bigcup_{|t|<\delta_1} \varphi_t( \{(x,\xi)\in S_H^*M\mid  |\tilde{b}(x,\xi')|>c/2\}
$$
In particular, by the elliptic parametrix construction and~\eqref{e:suppa} there are $e_i\in C_c^\infty(T^*M)$, $i=1,2,3$ such that
\begin{multline}
\label{e:squirrel3}
Op_h(a)=Op_h(e_1)b^+(x,hD_{x'})(hD_{x_n}-\tilde{\Lambda}_-)\\+Op_h(e_1)b^-(x,hD_{x'})(hD_{x_n}+\Lambda_+)
+Op_h(e_2)(-h^2\Delta_g-1)+O(h^\infty)_{L^2\to L^2}.
\end{multline}
Therefore, by~\eqref{e:squirrel1},~\eqref{e:octopus1}, and~\eqref{e:squirrel3}
\begin{multline*}
\|Op_h(a)u\|_{L^2}\leq C\|Op_h(\tilde{b})(hD_{x_n}+\Lambda_+)u\|_{L^2(H)}+C\|Op_h(\tilde{b})(hD_{x_n}-\tilde{\Lambda}_-)u\|_{L^2(H)}\\+Ch^{-1}\|(-h^2\Delta_g-1)u\|_{L^2}+O(h^\infty)\|u\|_{L^2}
\end{multline*}
Finally, note that 
\begin{equation*}
\begin{aligned}
\|Op_h(\tilde{b})(hD_{x_n}+\Lambda_+)u\|_{L^2(H)}&\leq \|Op_h(\tilde{b})hD_{\nu}u\|_{L^2(H)}+\|Op_h(\tilde{b})\Lambda_+u\|_{L^2(H)}\\
&\leq  \|Op_h(b_0)hD_{\nu}u\|_{L^2(H)}+\|Op_h(b_0)u\|_{L^2(H)}+O(h^\infty)\|u\|_{L^2(H)}\\
&\leq  \|Op_h(b_0)hD_{\nu}u\|_{L^2(H)}+\|Op_h(b_0)u\|_{L^2(H)}\\
&\qquad +O(h^\infty)(\|(-h^2\Delta_g-1)u\|_{L^2}+\|u\|_{L^2})
\end{aligned}
\end{equation*}
where in the next to last line we use that $\supp\tilde{b}\subset \supp b_0$ and in the last line, we use that Sobolev embedding.
Similarly,
\begin{equation*}
\begin{aligned}
\|Op_h(\tilde{b})(hD_{x_n}-\tilde{\Lambda}_-)u\|_{L^2(H)}
&\leq  \|Op_h(b_0)hD_{\nu}u\|_{L^2(H)}+\|Op_h(b_0)u\|_{L^2(H)}\\
&\qquad +O(h^\infty)(\|(-h^2\Delta_g-1)u\|_{L^2}+\|u\|_{L^2}),
\end{aligned}
\end{equation*}
which completes the proof.
\end{proof}

\begin{proof}[Proof of Theorem~\ref{t:CD}]
Suppose~\eqref{e:claim} does not hold. Then,
$$
\lim_{h\to 0}\|Op_h(b_0)u|_{H}\|_{L^2(H)}+\|Op_h(b_0)h\partial_\nu u|_{H}\|_{L^2(H)}=0.
$$
In particular, by Lemma~\ref{l:propagate}, there is $a\in C_c^\infty(T^*M)$ with $a \cap S^*M\neq 0$ such that 
$$
\lim_{h\to 0}\|Op_h(a)u\|_{L^2(M)}=0.
$$
But, this contradictions the results of ~\cite{DJ17,DJN19}.

To prove the second claim, observe that by unique continuation, see e.g.~\cite[Theorem 1.7]{GL17}, there is $c>0$ such that for all $h>0$, 
\begin{equation}
\label{e:uniqueContinue}
\|u\|_{L^2(M)}<Ce^{C/h}(\|u|_{H}\|_{L^2(U)}+\|h\partial_\nu u|_{H}\|_{L^2(U)}).
\end{equation}
Combining~\eqref{e:uniqueContinue} with~\eqref{e:claim} proves~\eqref{e:lowerBound}.
\end{proof}




\end{document}